\documentclass[12pt,a4paper,oneside]{amsart}

\ifdefined\SMART
\usepackage[paperwidth=12cm,paperheight=24cm,left=5mm,right=5mm,top=10mm,bottom=5mm]{geometry}
\else
\calclayout
\fi

\usepackage{color,amssymb,latexsym,amsfonts,textcomp,fancyhdr,calc,graphicx}
\usepackage{pdfpages,amsfonts,amsthm,amssymb,latexsym,graphicx,leftidx,fancyhdr}
\usepackage{amsmath,amstext,amsthm,amssymb,amsxtra}
\usepackage[usenames,dvipsnames,svgnames,table,rgb]{xcolor}
\usepackage[allcolors=blue,allbordercolors=blue,pdfborderstyle={/S/U/W 1}]{hyperref}
\usepackage[ocgcolorlinks]{ocgx2}
\usepackage{dsfont}
\usepackage[framemethod=tikz]{mdframed}
\usepackage{asymptote}
\usepackage{enumerate}

\usepackage{cite}

\usepackage{txfonts,pxfonts,tikz} %also txfonts 
\usepackage{tgschola}
\usepackage[T1]{fontenc}
\newtheorem{theorem}{Theorem}[section]
\newtheorem{corollary}{Corollary}[section]
\newtheorem{lemma}{Lemma}[section]

\newtheorem{remark}{Remark}[section]

\theoremstyle{definition}
\newtheorem{definition}{Definition}[section]

\newcommand{\beql}[1]{\begin{equation}\label{#1}}
\newcommand{\eeq}{\end{equation}}
\newcommand{\comment}[1]{}
\newcommand{\Ds}{\displaystyle}

\newcommand{\Abs}[1]{{\left|{#1}\right|}}

\newcommand{\Set}[1]{{\left\{{#1}\right\}}}

\newcommand{\RR}{{\mathbb R}}

\newcommand{\CC}{{\mathbb C}}

\newcommand{\one}{{\bf 1}}

\newcommand{\ft}[1]{\widehat{#1}}

\newcommand{\Arg}{{\rm Arg\,}}
\renewcommand{\Re}{{\rm Re\,}}
\renewcommand{\Im}{{\rm Im\,}}

\newcounter{rem}
\newcounter{step}
\newcounter{mysec}
\newcounter{mysubsec}[mysec]
\title[Curves in the Fourier zeros of polytopal regions]{Curves in the Fourier zeros of polytopal regions and the Pompeiu problem}

\author{Mihail N. Kolountzakis}
\address{\href{http://math.uoc.gr/en/index.html}{Department of Mathematics and Applied Mathematics}, University of Crete,\\Voutes Campus, 70013 Heraklion, Greece,\\and\\ \href{https://ics.forth.gr/}{Institute of Computer Science}, Foundation of Research and Technology Hellas, N. Plastira 100, Vassilika Vouton, 700 13, Heraklion, Greece}
\email{kolount@gmail.com}

\author{Emmanuil Spyridakis}
\address{\href{http://math.uoc.gr/en/index.html}{Department of Mathematics and Applied Mathematics}, University of Crete,\\Voutes Campus, 70013 Heraklion, Greece.}
\email{manos.ch.spyridakis@gmail.com}

\date{April 4, 2024}

\begin{document}

\begin{abstract}
We prove that any finite union $P$ of interior-disjoint polytopes in $\RR^d$ has the Pompeiu property, a result first proved by Williams \cite{williams1976partial}. This means that if a continuous function $f$ on $\RR^d$ integrates to 0 on any congruent copy of $P$ then $f$ is identically 0. By a fundamental result of Brown, Schreiber and Taylor \cite{brown1973spectral} this is equivalent to showing that the Fourier--Laplace transform of the indicator function of $P$ does not vanish identically on any $0$-centered complex sphere in $\CC^d$. Our proof initially follows the recent one of Machado and Robins \cite{machado2023null} who are using the Brion--Barvinok formula for the Fourier--Laplace transform of a polytope. But we simplify this method considerably by removing the use of properties of Bessel function zeros. Instead we use some elementary arguments on the growth of linear combinations of exponentials with rational functions as coefficients. Our approach allows us to prove the non-existence of complex spheres of any center in the zero-set of the Fourier--Laplace transform. The planar case is even simpler in that we do not even need the Brion--Barvinok formula. We then go further in the question of which sets can be contained in the null set of the Fourier--Laplace transform of a polytope by extending results of Engel \cite{engel2023identity} who showed that rationally parametrized hypersurfaces, under some mild conditions, cannot be contained in this null-set. We show that a rationally parametrized \textit{curve} which is not contained in an affine hyperplane in $\CC^d$ cannot be contained in this null-set. Results about curves parametrized by meromorphic functions are also given.
\end{abstract}


%\sloppy

\maketitle

\tableofcontents

\section{Introduction}

The Pompeiu problem \cite{pompeiu1929propriete,pompeiu1929propriete-a,zalcman1980offbeat}
is to determine if a bounded measurable subset $E \subseteq \RR^d$ has the \textit{Pompeiu property}:
\begin{definition}
The measurable set $E \subseteq \RR^d$ has the \textit{Pompeiu property} if the only continuous function $f$ on $\RR^d$ whose integrals on all congruent copies of $E$ vanish is the zero function.
\end{definition}
If $D$ is the unit ball in $\RR^d$ with indicator function $\one_D$ then the Fourier transform of $\one_D$, i.e., $\ft{\one_D}(\xi)=\int_D e^{-2\pi i \xi \cdot x}\,dx$, for $\xi \in \RR^d$, has rotational symmetry, is real-valued and it does have at least one zero. If $A=(a, 0, \ldots, 0)$, with $a > 0$, is such a zero it follows that the integrals of the function $f(x) = e^{-2\pi i A \cdot x}$ on every translate of $D$ are 0, without $f$ being 0. So $D$ (and any ball) does not have the Pompeiu property.

It has been conjectured that the ball is the only bounded convex body in $\RR^d$ that does not have the Pompeiu property. In \cite{williams1976partial} it is even conjectured that the ball is the only body homeomorphic to the ball that does not have the Pompeiu property. These conjectures are still open, but several large classes of sets are known which do have the Pompeiu property \cite{zalcman1992bibliographic}.

A very important result \cite{brown1973spectral,williams1976partial} regarding this problem is that a bounded measurable set $E \subseteq \RR^d$ does not have the Pompeiu property if and only if the Fourier--Laplace transform of its indicator function $\one_E$
\beql{fl}
\ft{\one_E}(z) = \int_{\RR^d} \one_E(x) e^{-2\pi i z\cdot x}\,dx,\ z \in \CC^d,
\eeq
does not vanish on any complex sphere $C_{0, R}$ (definition follows).
Notice that since $E$ is bounded the function $\ft{\one_E}(z)$ is entire.

\noindent{\bf Notation:} In this paper the inner product $x\cdot y$ of two vectors in $\RR^d$ or $\CC^d$ is always the quantity $x\cdot y = x_1 y_1 + x_2 y_2 + \cdots + x_d y_d$ (no conjugation).

\begin{definition}
A \textit{complex sphere} (\textit{complex circle} if $d=2$) of center $a=(a_1,\ldots,a_d) \in \CC^d$ and radius $R \in \CC\setminus\Set{0}$ is the subset of $\CC^d$
\beql{sphere}
C_{a, R} = \Set{z=(z_1,\ldots,z_d) \in \CC^d: (z_1-a_1)^2+\cdots+(z_d-a_d)^2=R^2}.
\eeq
\end{definition}

\begin{remark}
Let us stress here that the ``radius'' $R$ of the complex spheres related to Pompeiu's problem by the results in \cite{brown1973spectral,williams1976partial} is allowed to be a (non-zero) \underline{complex} number. Such a complex sphere $C_{0,R}$ contains no points in $\RR^d$ if $R \notin \RR$. It may be possible to eliminate the case of complex radii in relation to the Pompeiu problem (see discussion in \cite{machado2023null}) but we do not have to do so in this paper.

The reason we are excluding the case $R = 0$ is that in this case the structure of the variety is very different. For instance, in dimension $d=2$ the variety $z_1^2+z_2^2=0$ consists of the two complex lines $z_1 = \pm i z_2$.

Let us finally observe that every complex sphere $C_{0,R}$ in $\CC^d$, $d\ge 2$, contains a complex circle of the same radius, centered at 0, in $\CC^2\times\Set{0}^{d-2}$. 
\end{remark}

In this paper we build on the recent approach in \cite{machado2023null} who used the Brion--Barvinok \cite{brion1988points,barvinok1992exponential} formula for the Fourier--Laplace transform of a convex polytope in $\RR^d$ in order to show that this Fourier--Laplace transform does not vanish on any complex sphere centered anywhere in $\CC^d$ and, therefore, that any such polytope has the Pompeiu property (only $0$-centered spheres matter for the Pompeiu problem).

Our innovation is that we do not use at all properties of the zeros of Bessel functions as is done in \cite{machado2023null}. This allows us to give a much simpler proof with a clear potential for generalization to other varieties on which the Fourier--Laplace transform of a polytope cannot vanish identically. (These varieties do not necessarily mean something for the Pompeiu problem.)

Our main theorem concerning the Pompeiu problem is the following:
\begin{theorem}\label{th:main}
The Fourier--Laplace transform of the indicator function of any finite union of bounded convex polytopes with disjoint interiors cannot vanish on any complex sphere of any center in $\CC^d$ and any non-zero radius in $\CC$. Therefore such a region has the Pompeiu property.
\end{theorem}

The vanishing set of the Fourier--Laplace transform of the indicator function of a domain $\Omega \subseteq \RR^d$ is a much-studied object of huge importance in analysis and geometry \cite{berenstein1980inverse,kobayashi4bounded,kolountzakis2004study,kolountzakis1999steinhaus}, though often it is only its part contained in $\RR^d$ that is studied. The possibility to exclude certain varieties from the zero set is therefore potentialy very useful and we anticipate that our approach will be useful to other problems as well.

This work is also related to the recent paper \cite{engel2023identity}, where a method similar to ours has already been given to show the non-existence of spheres in the the Fourier zeros of polytopes. In \cite{engel2023identity} the author is mainly interested in identifying two polytopal regions (finite unions of interior-disjoint convex polytopes) whose Fourier--Laplace Transform is identical on some subset $E$ of $\CC^d$ (see our Corollary \ref{cor:identity}). The results obtained therein concern the case where $E$ is a rationally parametrized hypersurface in $\CC^d$ satisfying certain mild conditions, and also imply that such a surface cannot be contained in the null set of the Fourier--Laplace transform of a polytopal region. Our results do not concern identifying two polytopal regions from the equality of their Fourier--Laplace transform on a subset (which is the main concern of \cite{engel2023identity}). We only deal with what sets can be contained in the null set of the Fourier--Laplace transform of a polytopal region. In contrast to \cite{engel2023identity} we, however, require vanishing of the Fourier--Laplace Transform only on a rationally parametrized \textit{curve} which is not contained in an affine hyperplane in $\CC^d$. (If the curve is real then it is enough that it is not contained in an affine hyperplane of $\RR^d$.) Our results also cover curves which are parametrized by higher order meromorphic functions and are not restricted to the rational case (though the cleanest results for polytopes are still in the rational case). We achieve this by invoking a theorem \cite[Corollary of Theorem 1]{brownawell1976algebraic} (see also our Theorem \ref{th:brownawell}) that uses the growth of entire functions in order to show linear independence with coefficients from entire functions of smaller order. In \cite{engel2023identity} the same theorem is reproved for the case of order 0 essentially.

\subsection{Structure of the paper}
It all comes down to the null set of functions of the form
\begin{equation}\label{expoly}
\sum_{j=1}^N g_j(z) e^{f_j(z)},\ \ z\in\CC,
\end{equation}
where the $f_j, g_j$ are analytic functions, either entire or meromorphic, and where the growth of $g_j$ is restricted in relation to the growth of the $f_j$. By the Brion--Barvinok formula \eqref{brion}, \eqref{vertex} the Fourier--Laplace transform of a convex polytope, when restricted on a rationally parametrized curve, is given by exactly such an expression, where the $g_j$ and $f_j$ are rational functions.

In \S\ref{sec:vanishing} we develop our main tools about the zero sets of functions of the form \eqref{expoly}.
In Theorem \ref{th:gamma} we first examine the simple case where $\gamma(t)$ is a complex circle and the $g_j$ are polynomials, which is followed by Theorem \ref{th:gamma-rational} where $\gamma(t)$ is still a circle but the coefficients are allowed to be rational functions on $\gamma(t)$. Theorem \ref{th:gamma} leads us to the Fourier--Laplace zeros of a measure consisting of multiples of arc-length on line segments not being able to vanish identically on a circle (Theorem \ref{th:segments}). This is essentially the Pompeiu property for polygonal regions (Corollary \ref{cor:pompeiu-2d}). In general dimension $d \ge 2$ it is Theorem \ref{th:gamma-rational} that leads to the Pompeiu property for polytopal regions (Theorem \ref{th:polytope}, which is one step before Theorem \ref{th:main}) via the Brion--Barvinok formula \eqref{brion}, \eqref{vertex}. We treat the case $d=2$ separately from the general case $d \ge 2$ as in this case the situation is much simpler and does not require the Brion--Barvinok formula.

In \S\ref{sec:curves} then we leave the realm of circles and extend our discussion to curves defined by a rational or meromorphic parametrization. We show in Theorem \ref{th:general} that such a curve $\gamma(t)$, parametrized by meromorphic functions, cannot be in the zero set of
\begin{equation}\label{lin-comb}
\sum_{j=1}^N \frac{q_j(z)}{p_j(z)} e^{-2\pi i v_j \cdot z},\ \ \ z \in \CC^d,
\end{equation}
where $q_j, p_j$ are polynomials in $\CC^d$ and $v_j \in \RR^d$,
unless some strong relations are satisfied by the parametrization functions of $\gamma(t)$ or the $q_j$ vanish identically on $\gamma(t)$. In Corollary \ref{cor:rational} we show that this implies that no function \eqref{lin-comb} can vanish identically on a rationally parametrized curve which is not contained in an affine hyperplane in $\CC^d$, unless the $q_j$ themselves all vanish on the curve. This in turn implies (Corollary \ref{cor:identity}) that the Fourier--Laplace transform of a polytopal region cannot vanish identically on a rationally parametrized curve which is not contained in an affine hyperplane. In \S\ref{sec:appl} we reprove our results about the absence of circles in the null set using the theorems developed in \S\ref{sec:curves}. We also exhibit a simple curve in $\CC^2$ which is not rationally parametrizable yet can also not be contained in the null set of functions of the form dealt with in Corollary \ref{cor:rational}.

%%%%%%%%%%%%%%%%%%%%%%%%%%%%%%%%%%%%%%%%%%%%%%%%%%%%%%
%%%%%%%%%%%%%%%%%%%%%%%%%%%%%%%%%%%%%%%%%%%%%%%%%%%%%%
%%%%%%%%%%%%%%%%%%%%%%%%%%%%%%%%%%%%%%%%%%%%%%%%%%%%%%

\section{Vanishing on circles and growth of analytic functions}\label{sec:vanishing}

\begin{lemma}
For complex circles ($d=2$) we have the parametrization
\begin{align}\label{circle}
C_{a, R} &= \{a+(R \cos t) e_1 + (R \sin t) e_2: t\in\CC\}\\
	 &= \{(a_1+R \cos t, a_2+R \sin t): t\in\CC\}\nonumber
\end{align}
where $e_1 = (1, 0), e_2 = (0, 1), a = (a_1, a_2)$.
\end{lemma}
\begin{proof}
It is clear that $C_{a,R}$ contains the right hand side of \eqref{circle}.
To prove the reverse containment it suffices to show that whenever $w_1^2+w_2^2=1$, with $w_1, w_2 \in \CC$, then there is $t \in \CC$ such that $w_1 = \cos t$, $w_2 = \sin t$.
For this it is enough to show that $\cos t$ is onto $\CC$ and this reduces to solving a quadratic (not satisfied by 0) to find $e^{it}$.

\end{proof}

The following lemma is an easy calculation and is the basis of the method we are using in this section of the paper.
\begin{lemma}\label{cos-up}
If $z = x+iy$ with $x$ fixed and $y \to +\infty$ then
\beql{cosine}
\Abs{\cos{z}} = \left(\frac12 + o(1)\right) e^y,\ \ \Arg{\cos{z}} = -x+o(1).
\eeq
And if $A, B \in \CC$ and $w = A \sin z + B \cos z$ then
\begin{align}\label{cosine-sine}
\Abs{w} &= \left(\frac12+o(1)\right) \Abs{B+iA} e^y,\\
\Arg{w} &= -x + \Arg(B+iA) + o(1).\nonumber
\end{align}
\end{lemma}
\begin{proof}
We have, for $z = x+iy$,
$$
\cos z = \frac{e^{iz} + e^{-iz}}{2} = \frac12 e^{y-ix} (1+e^{-2y+2ix}),
$$
from which \eqref{cosine} is immediate for $x$ held fixed and $y \to +\infty$.

As for \eqref{cosine-sine}, simple calculation shows:
$$
w= e^{-ix+y}\left(\frac{B+iA}{2} +\frac{\left(B-iA\right)e^{-2(y+ix)}}{2}\right)
$$
and, again, fixing $x\in \RR$ and letting $y\to +\infty$ we obtain \eqref{cosine-sine}.

\end{proof}

We will use Lemma \ref{cos-up} in proving that certain linear combinations of exponential functions with polynomial or rational function coefficients cannot vanish on certain varieties of $\CC^d$. The first, easier case, which already exhibits the basic method, is the case of polynomial coefficients. (See also \cite{kolountzakis2017measurable} for the case of constant coefficients.)
\begin{theorem}\label{th:gamma}
Let $V = \Set{v_1, \ldots, v_N} \subseteq \RR^{{d}}$ be a finite set of points in 
{$\RR^d$, $d\ge 2$, such that their orthogonal projections onto $\RR^2\times\Set{0}^{d-2}$ are all different}
and let $p_j(x)$ be $N$ polynomials in $x \in \CC^d$. Let also
$$
\gamma(t) = a+R\cos{t}\, e_1 + R\sin{t}\, e_2,\ \ \ 0 \le t < 2\pi,
$$
denote a curve (here $R \in \CC\setminus\Set{0}, a\in\CC^d$ and $e_1 = (1, 0, \ldots, 0), e_2 = (0, 1, 0, \ldots, 0)$).
Suppose finally that the function
\beql{phi}
\phi(t) = \sum_{j=1}^N p_j(\gamma(t)) e^{-2\pi i v_j \cdot \gamma(t)}
\eeq
vanishes identically in $t \in [0, 2\pi)$. Then all $p_j(\gamma(t))$ vanish identically in $t \in \CC$.
\end{theorem}

\begin{proof}
Observe first that if we translate the set $V$ by $\tau \in \RR^d$ the function $\phi$ gets multiplied by $e^{-2\pi i \tau \cdot \gamma(t)}$, so the zeros of $\phi$ are not altered. This allows us to assume that one point of $V$ (say $v_1$) satisfies:
\begin{equation}
\Abs{v_1\cdot e_1 + i(v_1 \cdot e_2)} > \Abs{v_j \cdot e_1 + i(v_j\cdot e_2)},
\end{equation}
for every $j>1$.
{In other words, there is a single orthogonal projection of the points $v_j \in \RR^d$ onto $\RR^2\times\Set{0}^{d-2}$ of largest Euclidean length.}

Next, by analytic continuation, we conclude that $\phi(t)=0$ for all $t \in \CC$.

Finally, we may remove from \eqref{phi} all the summands for which $p_i$ vanishes identically on the complex circle $\gamma(\CC)$ and assume, contrary to what we want to prove, that at least one term remains in \eqref{phi}. We will obtain a contradiction.

The $j$-th exponent in \eqref{phi} is
$$
-2\pi i v_j \cdot \gamma(t) = -2\pi i (v_j\cdot a + Rv_j \cdot e_1  \cos{t} + R v_j \cdot e_2 \sin{t}),
$$
so that, as $\Re{t}$ is held fixed and $\Im{t} \to +\infty$, we have, by Lemma \ref{cos-up},
$$
\Abs{-2\pi i v_j \cdot \gamma(t)} = \left(\pi \big|R\big|+o(1)\right)\big|v_j \cdot e_1 + i(v_j\cdot e_2)\big|\, e^{\Im{t}}
$$
and
$$
\Arg(-2\pi i v_j \cdot \gamma(t)) = -\Re{t} + \Arg(-2\pi i(Rv_j\cdot e_1 +i(Rv_j \cdot e_2))+o(1).
$$
Fixing $\Re{t} = \Arg(-2\pi i(Rv_1\cdot e_1 +i(Rv_1 \cdot e_2))$ we achieve that for large enough $\Im{t}$
$$
\Re(-2\pi i v_1 \cdot \gamma(t)) \ge \left(\pi \big|R\big|+o(1)\right)\big|v_1 \cdot e_1 + i(v_1\cdot e_2)\big| e^{\Im{t}},
$$
while, at the same time, for $j \ge 2$ we have
$$
\Re(-2\pi i v_j \cdot \gamma(t)) \le {(1-\epsilon)}\left(\pi \big|R\big|+o(1)\right) \big|v_1 \cdot e_1 + i(v_1\cdot e_2)\big|e^{\Im t},
$$
for some $\epsilon>0$.

And by the following Lemma \ref{lm:decay} the first term in \eqref{phi}, whose exponential factor grows doubly exponentially in $\Im t$ and dominates all the others, is multiplied by $p_1(\gamma(t))$, a polynomial in $\cos{t}$, $\sin{t}$, which does not vanish identically and which can only affect the doubly exponential growth by an exponential.
\begin{lemma}\label{lm:decay}
{If $z \in \CC$ tends to infinity along a straight line,} that is $z = b+ta$, with $a, b \in \CC, a \neq 0, t \in \RR$, and $t \to +\infty$ then any exponential sum
$$
S(z) = \sum_{j=1}^N c_j e^{\mu_j z},\ \ c_j, \mu_j \in \CC,
$$
{which is not identically 0 in $t \in \RR$}
cannot decay more than exponentially in $t$. In other words, for some $c \in \RR$
$$
\limsup_{t \to +\infty} e^{c t} \Abs{S(b+ta)} > 0.
$$
\end{lemma}

\begin{proof}
We can absorb the constant $b$ into the coefficients $c_j$ of $S(z)$ so we can assume that $z = ta$. We can of course assume $a=1$ {(by replacing the $\mu_j$ with $a\mu_j$)} so we are looking at the function
$$
S(t) = \sum_{j=1}^N c_j e^{\mu_j t} = \sum_{j=1}^N c_j e^{t \Re \mu_j} e^{i t \Im \mu_j}.
$$
This sum is dominated by the terms for which $\Re \mu_j$ is maximal.
Let us say that this happens for $j \in J \subseteq \Set{1, \ldots, N}$ and assume {(possibly renumbering)}  that $1 \in J$.
Collecting these terms together their sum can be written as
$$
{T(t) =\ } e^{t \Re \mu_1} \sum_{j \in J} c_j e^{i t \Im \mu_j},
$$
{and we have
\beql{control}
C_1 \Abs{T(t)} \le \Abs{S(t)} \le C_2 \Abs{T(t)}
\eeq
for two positive constants $C_1, C_2$ that do not depend on $t$.
In particular \eqref{control} implies that $T(t)$ does not vanish identically since $S(t)$ does not.}

The trigonometric polynomial
$$
q(t) = \sum_{j \in J} c_j e^{i t \Im \mu_j}
$$
is {a non-zero} almost periodic function so $L = \limsup_{t \to \infty}\Abs{q(t)} > 0$. This implies that
$$
\limsup_{t\to\infty} e^{-t \Re\mu_1} \Abs{{T}(t)} > 0,
$$
{and the same is true for $S(t)$ due to \eqref{control}}.
\end{proof}

{Lemma \ref{lm:decay} together with the clear fact that all factors $p_j(\gamma(t))$ can grow at most exponentially in $t$}
implies that the term for $j=1$ is dominant in \eqref{phi} so that the vanishing of \eqref{phi} is impossible,
a contradiction.

\end{proof}

Working similarly we can prove the following which allows for rational coefficients.

\begin{theorem}\label{th:gamma-rational}
Let $V = \Set{v_1, \ldots, v_N} \subseteq \RR^d$ be a finite set of points in $\RR^d$,
{$d\ge 2$, such that their orthogonal projections onto $\RR^2\times\Set{0}^{d-2}$ are all different}
and let $p_j(x)$, {$q_j(x)$, $j=1, 2, \ldots, N$,} be $2N$ polynomials in $x \in \CC^d$. Let also
$$
\gamma(t) = a+R\cos{t}\, e_1 + R\sin{t}\, e_2,\ \ \ 0 \le t < 2\pi,
$$
denote a complex circle in $\CC^d$.
Here {$R \in \CC\setminus\Set{0}$}, $a\in\mathbb{C}^d$ and
{$e_1=(1, 0, \ldots, 0), e_2=(0, 1, 0, \ldots, 0)$.}
Assume that none of the functions $p_j(\gamma(t)), {q_j(\gamma(t))}$, $j=1, 2, \ldots, N$, vanish identically for $t \in [0, 2\pi)$, so that the function
\begin{equation}\label{phi-den}
\phi(t) = \sum_{j=1}^N \frac{{q_j(\gamma(t))}}{p_j(\gamma(t))} e^{-2\pi i v_j \cdot \gamma(t)},\ \ \ 0\le t < 2\pi,
\end{equation}
is defined for all but finitely many points in $[0, 2\pi)$.
Then $\phi(t)$ cannot vanish identically in $t$.
\end{theorem}

\begin{proof}
Assume that $\phi(t)$ vanishes for all $t \in [0, 2\pi)$ at which it is defined (all but finitely many points).

Observe first that if we translate the set $V$ by $\tau \in \RR^d$ the function $\phi$ gets multiplied by $e^{-2\pi i \tau \cdot \gamma(t)}$, so the zeros of $\phi$ are not altered. This allows us to assume that one point of $V$ (say $v_1$) satisfies:
\begin{equation}
\big|v_1\cdot e_1 + i(v_1 \cdot e_2) \big|>\big|v_j \cdot e_1 + i(v_j\cdot e_2) \big|
\end{equation}
for every $j>1$.
{(In other words, there is a single orthogonal projection of the points $v_j$ onto $\RR^2\times\Set{0}^{d-2}$ of largest Euclidean length.)}

The functions $\Ds \frac{{q_j(\gamma(t))}}{p_j(\gamma(t))}$ are meromorphic since the denominators are polynomials of $\cos t, \sin t$, hence $\phi(t)$ is defined by \eqref{phi-den} for all $t \in \CC$ for which the denominators do not vanish, i.e.\ with the exception of a countable set $Z$.
By analytic continuation, we may assume that $\phi(t)=0$ for all $t \in \CC\setminus Z$.

The $j$-th exponent in \eqref{phi-den} is
$$
-2\pi i v_j \cdot \gamma(t) = -2\pi i (v_j\cdot a + Rv_j \cdot e_1  \cos{t} + R v_j \cdot e_2 \sin{t}),
$$
so that, as $\Re{t}$ is held fixed and $\Im{t} \to +\infty$, we have, by Lemma \ref{cos-up},
$$
\Abs{-2\pi i v_j \cdot \gamma(t)} = (\pi \big|R\big|+o(1)) \cdot \big|v_j \cdot e_1 + i(v_j\cdot e_2)\big| e^{\Im{t}}
$$
and
$$
\Arg(-2\pi i v_j \cdot \gamma(t)) = -\Re{t} + \Arg(-2\pi i(Rv_j\cdot e_1 +i(Rv_j \cdot e_2))+o(1).
$$
Fixing $\Re{t} = \Arg(-2\pi i(Rv_1\cdot e_1 +i(Rv_1 \cdot e_2))$ we achieve that for large enough $\Im{t}$
\beql{growth-1a}
\Re(-2\pi i v_1 \cdot \gamma(t)) \ge (\pi \big|R\big|+o(1))\big|v_1 \cdot e_1 + i(v_1\cdot e_2)\big| e^{\Im{t}},
\eeq
while, at the same time, for $j \ge 2$ we have
\beql{growth-2a}
\Re(-2\pi i v_j \cdot \gamma(t)) \le {(1-\epsilon)}(\pi \big|R\big|+o(1)) \big|v_1 \cdot e_1 + i(v_1\cdot e_2)\big|e^{\Im t},
\eeq
for some $\epsilon>0$.

Since the ${q_j}, p_j$ are polynomials by Lemma \ref{lm:decay} the growth {or decay} at infinity of ${q_j(\gamma(t))}, p_j(\gamma(t))$ is at most exponentially fast. Since the growth of the exponential terms in \eqref{phi-den} is doubly exponential as shown in \eqref{growth-1a} and \eqref{growth-2a}, the coefficients in \eqref{phi-den} cannot compensate and there is one dominant term, the one corresponding to $v_1$, which cannot be killed by all the others combined, a contradiction.

\end{proof}

%%%%%%%%%%%%%%%%%%%%%%%%%%%%%%%%%%%%%%%%%%%%%%%%%%%%%%
%%%%%%%%%%%%%%%%%%%%%%%%%%%%%%%%%%%%%%%%%%%%%%%%%%%%%%
%%%%%%%%%%%%%%%%%%%%%%%%%%%%%%%%%%%%%%%%%%%%%%%%%%%%%%

\subsection{No circles in the Fourier zeros of polytopes}\label{sec:apps}

Let us start with the Pompeiu problem in dimension 2. This case is simpler than the case of general dimension since we do not need the Brion--Barvinok formula \eqref{brion}, \eqref{vertex}.
\begin{theorem}\label{th:segments}
Suppose $I_j$ are different straight line segments in $\RR^2$, $j=1,2,\ldots,N$, and $c_1,\ldots, c_N \in \CC\setminus\Set{0}$. Let $\mu$ be the measure $\mu = \sum_{j=1}^N c_j \delta_{I_j}$, where $\delta_{I_j}$ is arc-length on $I_j$. Then $\ft{\mu}$ cannot vanish identically on any {complex circle in $\CC^2$}. 
\end{theorem}
\begin{proof}
Let the unit vectors $u_1, \ldots, u_K \in \RR^2$, $K \le N$, be all the \textit{different} directions of the $I_j$ and apply the differential operator
$$
D = \partial_{u_1}\ldots\partial_{u_K}
$$
to $\mu$, which we view as a tempered distribution. Suppose the line segment $I$ has endpoints $a$ and $b$, and is of direction $u_1$ (from $a$ to $b$). Then
$$
D\delta_{I} = \partial_{u_2} \cdots \partial_{u_K} (\delta_a - \delta_b)
 = \partial_{u_2} \cdots \partial_{u_K} \delta_a - \partial_{u_2} \cdots \partial_{u_K} \delta_b,
$$
and
\begin{align}\label{ft-I}
\ft{D \delta_I}(x) &= (-2\pi i)^{K-1} (x \cdot u_2) \cdots (x \cdot u_K)\, e(-a \cdot x) \\
 &\ \ \ \ \  - (-2\pi i)^{K-1} (x \cdot u_2) \cdots (x \cdot u_K) \, e(-b \cdot x). \nonumber
\end{align}
(Here $e(x) = e^{2\pi i x}$.)

Summing over the different line segments in $\mu$ we obtain for $\nu = D \mu$
\beql{pv}
\ft{\nu}(x) = \sum_{v \in V} p_v(x) e(-v\cdot x),
\eeq
where $V$ is the set of endpoints of the $I_j$ (once each) and $p_v(x)$ is the polynomial which arises if we add up (with the corresponding coefficients) all terms arising in the corresponding equations \eqref{ft-I}  over all segments $I_j$ that have $v$ as a vertex. For every occurence of equation \eqref{ft-I} we have $K-1$ non-collinear unit vectors. This implies that $p_v(x)$ is always homogeneous of degree $K-1$. {(We omit from \eqref{pv} those $v$ for which $p_v(x)$ is the zero polynomial.)} Notice also that $\nu$ is not the zero distribution as the differential operator $D$ translates to multiplication by a polynomial on the Fourier side and cannot kill $\mu$ since $\ft{\mu}$, a continuous function on $\RR^2$, is not supported on subspaces and the zeros introduced by $D$ are a finite union of straight lines in $\RR^2$.

Assume now that $\ft{\nu}(x)$ vanishes on all points
{
\begin{align*}
\gamma(t) &= a + R \cos t\, e_1 + R \sin t\, e_2\\
 &= (a_1+R\cos t, a_2+R \sin t),\ \ \ 0 \le t < 2\pi,
\end{align*}
for some $R \in \CC\setminus\Set{0}, a = (a_1, a_2) \in \CC^2$.}
By Theorem \ref{th:gamma} all $p_v(\gamma(t))$ must vanish identically in $t$. 
{But $p_v(x)$ is a homogeneous polynomial of two variables, so, by the fundamental theorem of algebra, it factors over $\CC$ as a product of linear factors
$$
p_v(x) = p_v(x_1, x_2) = \prod_{j=1}^{K-1} (c_j x_1 + d_j x_2),
$$
so $\gamma(\CC) \subseteq \Set{z\in\CC^2: p_v(z)=0}$ should be contained in a union of complex lines in $\CC^2$, which clearly it is not, as this would imply a linear relation between $\cos t$ and $\sin t$.
}

We have proved that all polynomials $p_v$ in \eqref{pv} are identically zero, which means that $\nu = D\mu \equiv 0$, and this is impossible as mentioned above.

\end{proof}

\begin{corollary}\label{cor:pompeiu-2d}
The Fourier--Laplace transform of the indicator function of any polygonal region (not necessarily convex or even connected) cannot vanish on a {complex circle. Therefore every such region has the Pompeiu property}.
\end{corollary}

\begin{proof}
If we differentiate the indicator function of this region along a direction which is not parallel to any of the sides we get a measure as in Theorem \ref{th:segments} and the zero set of the Fourier Transform only increases {with the differentiation}.

\end{proof}

\begin{remark}
Corollary \ref{cor:pompeiu-2d} is true even of the Fourier--Laplace transform of any function that is locally constant on such a region (i.e. the level sets of this function are polygonal regions), not necessarily equal to 1 everywhere in the region as the indicator function is.
\end{remark}

{For general dimension $d\ge 2$ we can now show that a polytopal region has the Pompeiu property.}
\begin{theorem}\label{th:polytope}
Let $P\subseteq \RR^d$ be a $d$-dimensional polytope and $e_1=(1, 0, \ldots, 0), e_2=(0, 1, 0, \ldots, 0)$.
Suppose that all points in $V(P)$, the set of vertices of $P$, project orthogonally onto different points of $\RR^2\times\Set{0}^{d-2}$. 
Let $N(P) {\subseteq \CC^d}$ be the null set of the Fourier--Laplace transform of the {indicator} function of the polytope $P$. Then, $N(P)$ does not contain any complex circle
\begin{align}\label{circle-parametrization}
C_{a, R} &= \{ a+R\cos t \cdot e_1 + R\sin t \cdot  e_2:\  t\in \CC)\}\\
 &\subseteq a+\CC^2\times\Set{0}^{d-2} \nonumber
\end{align}
for any $a\in\CC^d,  R\in \CC \setminus \{0\}$.

{The same is true if $P$ is a finite union of interior-disjoint polytopes.
}
\end{theorem}
\begin{proof}\label{th:polytope-proof}
The proof follows from Theorem \ref{th:gamma-rational}, using the Brion--Barvinok formula for the Fourier--Laplace Transform of the indicator function of a $d$-dimensional polytope. Given a $d$-dimensional polytope $P$ of $\RR^d$ with a vertex set $V(P)$, one can define for each element of $V(P)$, call  it $v$, its tangent cone, denote it $K_v$,  as :
\begin{equation}
K_v := \{ v + \lambda(x-v) | x\in P, \lambda \ge 0\}.
\end{equation}
This is a pointed cone with apex $v$ and it has a set of generators, call them $w_1 ^v,...,w_m ^v$, so that it can also be written as $K_v = \{ v+ \lambda_1 w_1^v +...+\lambda_m w_m^ v | \lambda_j \ge 0 \}$. Each $w_k^v $ is a $1$-dimensional edge of $P$, emanating from $v$. When $m=d$, we say that the cone is simplicial and so we can define its determinant as:
\begin{equation}
\det K_v := | \det (w_1^v,...w_d^v) |.
\end{equation}
It is also known that every pointed cone can be triangulated into simplicial cones with no new generators, which means a collection $K_{v,1},..,K_{v,M_v}$ of simplicial cones with disjoint interiors such that $K_v = \bigcup_{j\le M_v} K_{v,j}$.
The Brion--Barvinok \cite{brion1988points,barvinok1992exponential,machado2023null} formula is:
\begin{equation}\label{brion}
\ft{\one}_P (z)=\sum_{v\in V(P)} \frac{q_v(z)}{p_v(z)}  e^{-2\pi i\, v \cdot z},
\end{equation}
whenever no denominator $p_v(z)$ vanishes,
where
\begin{equation}\label{vertex}
\frac{q_v(z)}{p_v(z)} = \sum_{j=1}^{M_v} \frac{\Abs{\det K_{v,j}}}{(2\pi i)^d \,
(w_{j,1}^v \cdot z) \cdots (w_{j,d}^v\cdot z)}.
\end{equation}
If $P$ is a finite union of interior-disjoint polytopes $P=\bigcup_{j=1}^J P_j$ then $\one_P = \sum_{j=1}^J \one_{P_j}$ and taking the Fourier--Laplace Transform we conclude that $P$ still satisfies \eqref{brion} with \eqref{vertex}, where now $V(P)$ is the totality of the vertices of the $P_j$, written once each, and $M_v$ is the total number of simplicial cones emanating from vertex $v$, over all polytopes that share $v$ as a vertex.

From the form of \eqref{vertex} it follows that the denominator $p_v(z)$ can be taken to be a product of linear factors of the form $w \cdot z$.
First we have to make sure that the denominators do not vanish identically on $C_{a,R}$. This is indeed true as such a vanishing would require that some $w\cdot z$ would vanish identically on $C_{a,R}$, where $w=(w_1, \ldots, w_d) \in \RR^d$ is one of the one-dimensional edges of the polytope (and a difference of two vertices of the polytope). This is equivalent to
$$
0 = w\cdot(a+R\cos t \cdot e_1+R\sin t \cdot e_2) = w\cdot a + R w_1 \cos t + R w_2 \sin t
$$
for all $t \in \CC$. By our assumption on the unique projection of the vertices of $P$ onto $\RR^2\times\Set{0}^{d-2}$ it follows that $w_1, w_2$ cannot both be 0, so this equation contradicts the linear independence over $\CC$ of the functions $1, \cos t, \sin t$.

At the same time, the numerator $q_v(z)$ can be taken to be a homogeneous polynomial.
Since we care about the vanishing of \eqref{brion} on $C_{a, R}$ we may discard all fractions in \eqref{vertex} for which $q_v(z)$ vanishes identically on $C_{a, R}$.
We assume that no term remains in \eqref{brion} and we arrive at a contradiction. By the homogeneity of both $q_v(z)$ and $p_v(z)$ it follows that $q_v(z)/p_v(z)$ vanishes on all points of $\CC C_{a, R}$ on which \eqref{brion} is valid, i.e., on all points out of the hyperplanes $w\cdot z = 0$ appearing in the denominators of \eqref{vertex}. Since these hyperplanes do not cover $C_{a, R}$, as explained in the previous paragraph,  it follows that 0 is an accumulation point of the zeros of $q_v(z)$ (for all $v$). By the continuity of $\ft{\one_P}(z)$ we obtain that this function vanishes at 0, a contradiction since $\ft{\one_P}(0) = \Abs{P}$ (the volume of $P$).

Thus the requirements of Theorem \ref{th:gamma-rational} (that all fractions appearing in \eqref{phi-den} do not vanish identically on $C_{a, R}$) are satisfied and we conclude that $\ft{\one_P}(z)$ does not vanish identically on $C_{a,R}$.

\end{proof}

We can now complete the proof of Theorem \ref{th:main}.

\begin{proof}[Proof of Theorem \ref{th:main}]\label{pf:main}
The Pompeiu property is invariant under orthogonal transformations so by applying  an appropriate orthogonal transformation to our set we may assume that all its vertices project orthogonally onto different points in $\RR^2\times\Set{0}^{d-2}$.
If our set does not have the Pompeiu property then the Fourier--Laplace transform of its indicator function must contain some complex sphere $C_{0,R}$. This complex sphere contains a complex circle in $\CC^2\times\Set{0}^{d-2}$. But this is impossible by Theorem \ref{th:polytope}.
\end{proof}

%%%%%%%%%%%%%%%%%%%%%%%%%%%%%%%%%%%%%%%%%%%%%%%%%%%%%%
%%%%%%%%%%%%%%%%%%%%%%%%%%%%%%%%%%%%%%%%%%%%%%%%%%%%%%
%%%%%%%%%%%%%%%%%%%%%%%%%%%%%%%%%%%%%%%%%%%%%%%%%%%%%%

\section{Curves parametrized by meromorphic functions}\label{sec:curves}

\begin{theorem}\label{th:general}
Let $V = \Set{v_1, \ldots, v_N} \subseteq \RR^d$ be a finite set of points in $\RR^d$,
$d\ge 2$,
and let $p_j(x)$, $q_j(x) \in \CC[x]$, $j=1, 2, \ldots, N$, be $N$ polynomials in $x \in \CC^d$. Let also
$$
\gamma(t) = (r_1(t), \ldots, r_d(t)),\ \ \ 0 \le t \le 1,
$$
denote a complex curve in $\CC^d$ parametrized by functions
$$
r_j(t)=\frac{a_j(t)}{b_j(t)},\ \ j=1, 2, \ldots, d,
$$
where $a_j(t), b_j(t)$ are entire functions of $t \in \CC$, and let $\rho\ge 0$ denote the maximum order of all $a_j(t), b_j(t)$.

Assume that none of the functions $p_j(\gamma(t)), q_j(\gamma(t))$, $j=1, 2, \ldots, N$, vanish identically for $t \in [0, 1]$, so that the function
\beql{phi-rat}
\phi(t) = \sum_{j=1}^N \frac{q_j(\gamma(t))}{p_j(\gamma(t))} e^{-2\pi i v_j \cdot \gamma(t)},\ \ \ 0\le t < 1,
\eeq
is defined for all but finitely many points in $[0, 1]$.

Then $\phi(t)$ cannot vanish identically in $t$ unless some non-trivial linear combination of the $r_j(t)$ (with complex coefficients) is a polynomial in $t$ of degree $\le \rho$.
\end{theorem}

\begin{proof}
We will use the following result with $d=1$.
\begin{theorem}[Corollary of Theorem 1 in \cite{brownawell1976algebraic}]\label{th:brownawell}
Let $f_1(z), \ldots, f_m(z)$ be meromorphic functions of $z \in \CC^d$.
Then $\exp(f_1), \ldots, \exp(f_m)$ are linearly dependent over the ring of entire functions of order $\le \rho$ if and only if, for some $1 \le k < l \le m$, $f_k(z)-f_l(z) \in \CC[z]$ with total degree at most $\rho$.
\end{theorem}
If the meromorphic function $\phi(t)$ vanishes identically for $t \in [0, 1]$ then it vanishes on all $t \in \CC$ that are not poles of some $p_j(\gamma(t))$ or some $r_k(t)$. Write
$$
\frac{q_j(\gamma(t))}{p_j(\gamma(t))} = \frac{Q_j(t)}{P_j(t)},
$$
with $Q_j(t), P_j(t)$ entire, of order $\le \rho$, having no common zeros. The entire function $P(t) = \prod_{j=1}^N P_j(t)$ can also not vanish identically on $[0, 1]$. Multiplying the identity $\phi(t) = 0$ by $P(t)$ we obtain, for some entire functions $\widetilde{Q_j}(t) \nequiv 0$ of order $\le \rho$,
\beql{poly-comb}
0 = \sum_{j=1}^N \widetilde{Q_j}(t)  e^{-2\pi i v_j \cdot \gamma(t)},\ \ t \in \CC.
\eeq
We now apply Theorem \ref{th:brownawell} to \eqref{poly-comb}. It follows that for some $1 \le k < l \le N$ we have that
$$
-2\pi i (v_k-v_l) \cdot \gamma(t) \text{ is a polynomial of degree $\le \rho$}.
$$
Writing $0 \neq w = -2\pi i (v_k-v_l)$ this means that
$$
w_1 r_1(t)+\cdots+w_d r_d(t) \text{ is a polynomial of degree $\le \rho$},
$$
as we had to show.

\end{proof}

In the important case of rationally parametrized curves we have the following.

\begin{corollary}\label{cor:rational}
Let $V = \Set{v_1, \ldots, v_N} \subseteq \RR^d$ be a finite set of points in $\RR^d$,
$d\ge 2$,
and let $p_j(x)$, $q_j(x) \in \CC[x]$, $j=1, 2, \ldots, N$, be $N$ polynomials in $x \in \CC^d$. Let also
$$
\gamma(t) = (r_1(t), \ldots, r_d(t)),\ \ \ 0 \le t \le 1,
$$
denote a complex curve in $\CC^d$ parametrized by rational functions $r_j(t)$, $j=1, 2, \ldots, d$, which is not contained in any affine hyperplane of $\CC^d$ (translate of some $(d-1)$-dimensional $\CC$-subspace).

Assume that none of the functions $p_j(\gamma(t)), q_j(\gamma(t))$, $j=1, 2, \ldots, N$, vanish identically for $t \in [0, 1]$, so that the function
\beql{phi-rat-cor}
\phi(t) = \sum_{j=1}^N \frac{q_j(\gamma(t))}{p_j(\gamma(t))} e^{-2\pi i v_j \cdot \gamma(t)},\ \ \ 0\le t < 1,
\eeq
is defined for all but finitely many points in $[0, 1]$.

Then $\phi(t)$ cannot vanish identically in $t$.
\end{corollary}

\begin{proof}
The $r_j(t)$ are quotients of polynomials which are entire functions of order 0, so, by Theorem \ref{th:general}, for some $1\le k < l \le d$ we must have
$$
-2\pi i (v_k-v_l)\cdot \gamma(t) = C,
$$
a constant. But this implies that $\gamma([0, 1])$ is contained in the affine hyperplane
$$
\Set{x \in \CC^d: (v_k-v_l)\cdot x = \frac{C i}{2\pi}},
$$
in contradiction to our assumption.
\end{proof}

The following result specializes Corollary \ref{cor:rational} to Fourier--Laplace transforms of polytopal regions.
\begin{corollary}\label{cor:identity}
Suppose $\gamma(t)$, $t \in [0, 1]$, is a curve parametrized by rational functions with complex coefficients, which is not contained in any affine hyperplane of $\CC^d$.
If $P$ is a polytopal region then its Fourier--Laplace transform cannot vanish identically on $\gamma(t)$, $t \in [0, 1]$.

In particular, the conclusion is true if $\gamma([0, 1]) \subseteq \RR^d$ and $\gamma([0, 1])$ is not contained in any (real) affine hyperplane of $\RR^d$.
\end{corollary}
\begin{proof}
As explained in the proof of Theorem \ref{th:polytope} and formulas \eqref{brion} and \eqref{vertex}
the Fourier--Laplace transform of a polytopal region $P$, is of the form
\begin{equation}\label{ft-form}
\sum_v \frac{q_v(z)}{p_v(z)} e^{-2\pi i v\cdot z}.
\end{equation}
Corollary \ref{cor:rational} tells us that this function cannot vanish on $\gamma([0, 1])$ as soon as we can prove that all $q_v(z), p_v(z)$ do not vanish identically on $\gamma([0, 1])$. There is no doubt about the $p_v(z)$ which, according to \eqref{vertex}, are products of linear factors and $\gamma([0, 1])$ is not contained in any hyperplane, by assumption. Then, as we did in the proof of Theorem \ref{th:polytope}, we throw away all summands where $q_v(z)$ vanishes identically on $\gamma([0, 1])$ and assume that nothing remains, in order to reach a contradiction.

Were the homogeneous polynomial $q_v(z)$ to vanish identically on $\gamma([0, 1])$ it would not be a constant, hence we would have $q_v(0) = 0$ and also $q_v(z)$ would vanish on $\CC\gamma([0, 1])$ minus the hyperplanes on which some $p_v(z)$ vanishes. By the continuity of the Fourier--Laplace transform at 0 we obtain then that $\ft{\one_P}(0)=0$ which is a contradiction since $\ft{\one_P}(0)=\Abs{P} > 0$. Therefore some summands in \eqref{ft-form} do remain which do not vanish identically on $\gamma([0, 1])$ and Corollary \ref{cor:rational} proves what we want.

To see the last remark about real curves, notice that if the curve $\gamma([0, 1]) \subseteq \RR^d$ is not contained in an affine subspace of $\RR^d$ then its difference set $\gamma([0, 1])-\gamma([0, 1])$ spans $\RR^d$ over the reals. But this implies that $\gamma([0, 1])-\gamma([0, 1])$ spans $\CC^d$ over the complex numbers so that $\gamma([0, 1])$ cannot be contained in an affine hyperplane of $\CC^d$.
\end{proof}

%%%%%%%%%%%%%%%%%%%%%%%%%%%%%%%%%%%%%%%%%%%%%%%%%%%%%%
%%%%%%%%%%%%%%%%%%%%%%%%%%%%%%%%%%%%%%%%%%%%%%%%%%%%%%
%%%%%%%%%%%%%%%%%%%%%%%%%%%%%%%%%%%%%%%%%%%%%%%%%%%%%%

\subsection{Some specific curves}\label{sec:appl}

\subsubsection{Circle}

Using Theorem \ref{th:general} we can give an alternative proof of Theorem \ref{th:polytope}. As in our previous proof on p.\ \pageref{th:polytope-proof} we deduce that we have
$$
\ft{\one_P}(z) = \sum_{v \in V(P)} \frac{q_v(z)}{p_v(z)} e^{-2\pi i v \cdot z},
$$
when no denominator $p_v(z)$ vanishes. Restricting $z_3 = \cdots = z_d = 0$ we want to show that the above function of $z_1, z_2$ does not vanish on a circle. We can use either the trigonometric parametrization \eqref{circle-parametrization} or the well-known rational parametrization
$$
C_{a,R} = \Set{a+\frac{1-t^2}{1+t^2}e_1+\frac{2t}{1+t^2}e_2: t \in \CC\setminus\Set{\pm i}}.
$$
We use the rational parametrization first. By Corollary \ref{cor:rational}, with $d=2$, and with the curve $\gamma(t)$
\begin{align*}
r_1(t) &= a_1 + \frac{1-t^2}{1+t^2},\\
r_2(t) &= a_2 + \frac{2t}{1+t^2},
\end{align*} being the circle in question we obtain our contradiction since $\gamma(\CC\setminus\Set{\pm i})$ is not contained in a one-dimensional affine subspace of $\CC^2$ (no linear combination of $r_1(t), r_2(t)$ is a constant).

If we use the trigonometric parametrization \eqref{circle-parametrization}, we conclude that some non-trivial linear combination of the functions
\begin{align*}
r_1(t) &= a_1 + R\cos{t},\\
r_2(t) &= a_2 + R\sin{t},
\end{align*} 
must equal a polynomial of $t$ of degree at most 1 (as the trigonometric functions are entire of order 1). Again this is impossible (the function $t$ is not a trigonometric polynomial).

\subsubsection{Not rationally parametrized}

Next we give an example of a curve in $\RR^2$ which cannot be rationally parametrized, yet can serve as a curve such that no Fourier--Laplace transform of a polytopal region can vanish identically on it. There are of course many other examples.

\begin{lemma}\label{lm:non-rp}
If $p(x, y) \in \CC[x, y]$ is not the zero polynomial then it cannot vanish identically on the curve
\beql{non-rp}
\gamma(t) = (t^2, \sin t),\ \ t\in [0, 1].
\eeq
\end{lemma}
\begin{proof}
Write $p(x, y) = \sum_{j=0}^N p_j(x) y^j$ for some $p_j(x) \in \CC[x]$.
Suppose $p(t^2, \sin t) = 0$ for all $t \in [0, 1]$. This leads to the equation
$$
0 = \sum_{j=0}^N p_j(t^2) \sin^j t, \ \ (\text{all } t \in \CC)
$$
for some polynomials $p_j(t) \in \CC[t]$. 
By a method similar to that used in the proof of Theorem \ref{th:gamma} we can prove that this implies that all $p_j(t)$ are identically 0. More concretely writing $t = -i s$ and letting $s \to +\infty$ we have
\begin{align*}
	0 &= \sum_{j=0}^N p_j(-s^2) \frac{1}{(2i)^j} (e^{s}-e^{-s})^j\\
	&= \sum_{j=0}^N p_j(-s^2) \frac{1}{(2i)^j} (e^{js}+O(e^{(j-1)s})),
\end{align*}
and from this, working successively from $j=N$ down to $j=0$, it easily follows that all $p_j \equiv 0$, which implies $p(x, y) \equiv 0$.
\end{proof}

The curve \eqref{non-rp} in $\RR^2$ cannot be rationally parametrized. If it could be rationally parametrized then it would also be an algebraic curve \cite[Chapter 3, \S 3, Implicitization]{cox2013ideals}, meaning that there exists a polynomial $0 \nequiv p(x, y) \in \CC[x, y]$ such that $p(t^2, \sin t) = 0$ for all $t$. From Lemma \ref{lm:non-rp} this implies that $p(x, y) \equiv 0$, a contradiction.

Suppose now that a function of the form \eqref{phi-rat} vanishes identically on $\gamma([0, 1])$. It follows from Lemma \ref{lm:non-rp} that none of the fractions $q_j(\gamma(t))/p_j(\gamma(t))$ in \eqref{phi-rat} vanish identically in $t$.
By Theorem \ref{th:general}, with $\rho=1$, it follows then that
$$
-2\pi i (v_k-v_l)\cdot \gamma(t) = At+B,
$$
for some $A, B \in \CC$, $k \neq l$, and for all $t \in [0, 1]$. By the linear independence of the functions $1, t, t^2, \sin t$ this implies that $A=B=0$, $v_k = v_l$, which contradicts the fact that all $v_j$ are distinct. So vanishing on this curve is impossible for such functions (and therefore for Fourier--Laplace transforms of indicator functions of polytopal regions).
This is a curve which is not covered by results in \cite{engel2023identity} as it is not rationally parametrizable.

We have shown the following.
\begin{theorem}\label{th:identity}
If $P$ is a polygonal region then its Fourier--Laplace transform cannot vanish identically on \eqref{non-rp}.
\end{theorem}

\bibliographystyle{alpha}
%%%%% replace with the right location for your system
\bibliography{../mk-bibliography.bib}

\end{document}